\documentclass{amsart}[12pt]
\usepackage[top=1.0in, bottom=1.0in, left=1.1in, right=1in]{geometry}
\usepackage{enumerate}
\usepackage[english]{babel}
\usepackage[arrow,matrix,curve]{xy}
\usepackage{hyperref}

\DeclareMathOperator{\res}{res}
\DeclareMathOperator{\tr}{tr}

\newcommand{\sm}{\wedge}
\newcommand{\iso}{\cong} 
\renewcommand{\to}{\longrightarrow}
\newcommand{\bs}{\backslash} 
\newcommand{\mC}{\mathbb C}
\newcommand{\bMU}{\mathbf{MU}}
\newcommand{\xra}{\xrightarrow}

\numberwithin{equation}{section}
\newtheorem{thm}[equation]{Theorem}
\newtheorem{prop}[equation]{Proposition}

\theoremstyle{definition}
\newtheorem{defn}[equation]{Definition}
\newtheorem{rk}[equation]{Remark}
\newtheorem{eg}[equation]{Example}
\newtheorem{con}[equation]{Construction}

\begin{document}

\title{Chern classes in equivariant bordism}

\author{Stefan Schwede}
\address{Mathematisches Institut, Universit\"at Bonn, Germany}
\email{schwede@math.uni-bonn.de}

\begin{abstract} We introduce Chern classes in $U(m)$-equivariant homotopical bordism
  that refine the Conner-Floyd--Chern classes in the $\bMU$-cohomology of $B U(m)$.
  For products of unitary groups, our Chern classes form regular sequences
  that generate the augmentation ideal of the equivariant bordism rings.
  Consequently, the Greenlees--May local homology spectral sequence collapses for products of unitary groups.
  We use the Chern classes to reprove the $\bMU$-completion theorem of Greenlees--May and La Vecchia.\medskip
\end{abstract}

\maketitle

\section*{Introduction}

Complex cobordism $\bMU$ is arguably the most important cohomology theory in algebraic topology.
It represents the bordism theory of stably almost complex manifolds,
and it is the universal complex oriented cohomology theory;
via Quillen's celebrated theorem \cite{quillen:formal_group},
$\bMU$ is the entry gate for the theory of formal group laws into
stable homotopy theory, and thus the cornerstone of chromatic stable homotopy theory.

Tom Dieck's homotopical equivariant bordism $\bMU_G$ \cite{tomDieck:bordism_integrality},
defined with the help of equivariant Thom spaces,
strives to be the legitimate equivariant refinement of complex cobordism,
for compact Lie groups $G$.
The theory $\bMU_G$ is the universal equivariantly complex oriented theory;
and for abelian compact Lie groups, the coefficient ring $\bMU_G^*$ carries the universal
$G$-equivariant formal group law \cite{hausmann:group_law}.
Homotopical equivariant bordism receives a homomorphism
from the geometrically defined equivariant bordism theory; due to the lack
of equivariant transversality, this homomorphism is {\em not} an isomorphism for nontrivial groups.
In general, the equivariant bordism ring $\bMU^*_G$
is still largely mysterious; the purpose of this paper is to elucidate its structure
for unitary groups, and for products of unitary groups.

Chern classes are important characteristic classes for complex vector bundles
that were originally introduced in singular cohomology.
Conner and Floyd \cite[Corollary 8.3]{conner-floyd:relation_cobordism}
constructed Chern classes for complex vector bundles in complex cobordism;
in the universal cases, these yield classes $c_k\in \bMU^{2 k}(B U(m))$
that are nowadays referred to as Conner--Floyd--Chern classes. 
Conner and Floyd's construction works in much the same way for any complex oriented cohomology theory,
see \cite[Part II, Lemma 4.3]{adams:stable_homotopy};
in singular cohomology, it reduces to the classical Chern classes.
The purpose of this note is to define and study Chern classes
in $U(m)$-equivariant homotopical bordism $\bMU^*_{U(m)}$
that map to the Conner--Floyd--Chern classes
under tom Dieck's bundling homomorphism \cite[Proposition 1.2]{tomDieck:bordism_integrality}.
Our classes satisfy the analogous formal properties as their classical counterparts,
including the equivariant refinement of the Whitney sum formula, see Theorem \ref{thm:CFC main}.
Despite the many formal similarities, there are crucial qualitative differences
compared to Chern classes in complex oriented cohomology theories:
Our Chern classes are {\em not} characterized by their restriction to the maximal torus,
and some of our Chern classes are zero-divisors, see Remark \ref{rk:torus_restriction}.

We will use our Chern classes and the splitting of \cite{schwede:split BU}
to prove new structure results  about the equivariant bordism rings $\bMU^*_{U(m)}$
for unitary groups, or more generally for products of unitary groups.
To put this into context, we recall that in the special case when $G$ is an {\em abelian} compact Lie group,
the graded ring $\bMU^*_G$ is concentrated in even degrees and free as a module
over the nonequivariant cobordism ring $\bMU^*$ \cite[Theorem 5.3]{comezana}, \cite{loeffler:equivariant},
and the bundling homomorphism $\bMU^*_G\to \bMU^*(B G)$ is completion
at the augmentation ideal of $\bMU^*_G$ \cite[Theorem 1.1]{comezana-may}, \cite{loeffler:bordismengruppen}.
For nonabelian compact Lie groups $G$, however, the equivariant bordism rings $\bMU^*_G$
are still largely mysterious.
\newpage

The main result of this note is the following:\smallskip

{\bf Theorem.} {\em Let $m\geq 1$ be a natural number.
  \begin{enumerate}[\em (i)]
  \item
    The sequence of Chern classes $c_m^{(m)},c_{m-1}^{(m)},\dots,c_1^{(m)}$
    is a regular sequence that generates the augmentation ideal of the graded-commutative ring $\bMU^*_{U(m)}$.
 \item
   The completion of $\bMU_{U(m)}^*$ at the augmentation ideal
   is a graded $\bMU^*$-power series algebra in the above Chern classes.
 \item
    The bundling homomorphism $\bMU_{U(m)}^*\to \bMU^*(B U(m))$ extends to an isomorphism
    \[ ( \bMU_{U(m)}^*)^\wedge_I \ \to \ \bMU^*(BU(m)) \]
    from the completion at the augmentation ideal.
  \end{enumerate}
}

We prove this result as a special case of Theorem \ref{thm:completions} below;
the more general version applies to products of unitary groups.
As we explain in Remark \ref{rk:degenerate}, the regularity of the Chern classes
also implies that the Greenlees--May local homology spectral sequence
converging to $\bMU^*(BU(m))$ degenerates
because the relevant local homology groups vanish in positive degrees.
As another application we use the Chern classes in equivariant bordism
to give a reformulation and self-contained proof of work of Greenlees--May \cite{greenlees-may:completion}
and La Vecchia \cite{lavecchia} on the completion theorem for $\bMU_G$,
see Theorem \ref{thm:completion}.

\section{Equivariant \texorpdfstring{$\bMU$}{MU}-Chern classes}

In this section we introduce the Chern classes in $U(m)$-equivariant homotopical bordism,
see Definition \ref{def:CFC}. We establish their basic properties
in Theorem \ref{thm:CFC main}, including a Whitney sum formula and the fact that the bundling homomorphism
takes our Chern classes to the Conner--Floyd--Chern classes in $\bMU$-cohomology.

We begin by fixing our notation.
For a compact Lie group $G$, we write $\bMU_G$ for the $G$-equivariant homotopical bordism spectrum
introduced by tom Dieck \cite{tomDieck:bordism_integrality}.
For our purposes, it is highly relevant that the theories $\bMU_G$ for varying compact Lie groups $G$ assemble
into a global stable homotopy type, see \cite[Example 6.1.53]{schwede:global}.
For an integer $n$, we write $\bMU_G^n=\pi_{-n}^G(\bMU)$ for the $G$-equivariant coefficient group
in cohomological degree $n$.

Since $\bMU$ comes with the structure of a global ring spectrum, it supports
graded-commutative multiplications on $\bMU_G^*$, as well as external multiplication pairings
\[ \times \ : \ \bMU_G^k\times \bMU_K^l \ \to \ \bMU_{G\times K}^{k+l} \]
for all pairs of compact Lie groups $G$ and $K$.
We write $\nu_k$ for the tautological representation
of the unitary group $U(k)$ on $\mC^k$; we denote its Euler class by
\[ e_k \ = \  e(\nu_k) \ \in \ \bMU^{2 k}_{U(k)}\ ,\]
compare \cite[page 347]{tomDieck:bordism_integrality}.
We write $U(k,m-k)$ for the block subgroup of $U(m)$ consisting of matrices of the form
$(\begin{smallmatrix}A & 0 \\ 0 & B \end{smallmatrix})$
for $(A,B)\in U(k)\times U(m-k)$.
We write $\tr_{U(k,m-k)}^{U(m)}:\bMU_{U(k,m-k)}^*\to\bMU_{U(m)}^*$
for the transfer associated to the inclusion $U(k,m-k)\to U(m)$,
see for example \cite[Construction 3.2.22]{schwede:global}.

\begin{defn}\label{def:CFC}
  For $0\leq k\leq m$, the {\em $k$-th Chern class}
  in equivariant complex bordism is the class
  \[ c_k^{(m)} \ = \ \tr_{U(k,m-k)}^{U(m)}(e_k\times 1_{m-k})\ \in \ \bMU^{2 k}_{U(m)}\ , \]
 where $1_{m-k}\in\bMU_{U(m-k)}^0$ is the multiplicative unit. We also set $c_k^{(m)} =0$ for $k>m$.
\end{defn}

In the extreme cases $k=0$ and $k=m$, we recover familiar classes:
Since $e_0$ is the multiplicative unit in the nonequivariant cobordism ring $\bMU^*$,
the class $c_0^{(m)}=1_m$ is the multiplicative unit in $\bMU_{U(m)}^0$. 
In the other extreme, $c_m^{(m)}=e_m=e(\nu_m)$ is the Euler class of
the tautological $U(m)$-representation.
As we will show in Theorem \ref{thm:CFC main} (ii), the classes $c_k^{(m)}$
are compatible in $m$ under restriction to smaller unitary groups.

\begin{rk}\label{rk:torus_restriction}
  We alert the reader that the restriction homomorphism 
  \[ \res^{U(m)}_{T^m}\ :\ \bMU^*_{U(m)}\ \to \ \bMU^*_{T^m} \]
  is not injective for $m\geq 2$, where $T^m$ denotes a  maximal torus in $U(m)$.
  So the Chern classes in $\bMU^*_{U(m)}$ are not characterized by their restrictions
  to the maximal torus -- in contrast to the nonequivariant situation for complex oriented cohomology theories.
  To show this we let $N$ denote the maximal torus normalizer inside $U(m)$. The class
  \[  1- \tr_N^{U(m)}(1) \ \in \ \bMU^0_{U(m)} \]
  has infinite order because the $U(m)$-geometric fixed point map
  takes it to the multiplicative unit; in particular, this class is nonzero.
  The double coset formula \cite[IV Corollary 6.7 (i)]{lms}
  \[ \res^{U(m)}_{T^m}(\tr_N^{U(m)}(1))\ = \ \res^N_{T^m}(1)\ = \ 1  \]
  implies that the class $ 1- \tr_N^{U(m)}(1)$ lies in the kernel of the restriction homomorphism
  $\res^{U(m)}_{T^m}:\bMU^0_{U(m)}\to \bMU^0_{T^m}$.

  Moreover, the Chern class $c_1^{(2)}$ is a zero-divisor in the ring $\bMU^*_{U(2)}$,
  also in stark contrast to Chern classes in complex oriented cohomology theories.
  Indeed, reciprocity for restriction and transfers \cite[Corollary 3.5.17 (v)]{schwede:global}
  yields the relation
  \begin{align*}
    c_1^{(2)}\cdot (1-\tr_N^{U(2)}(1)) \
    &= \ \tr_{U(1,1)}^{U(2)}(e_1\times 1)\cdot (1-\tr_N^{U(2)}(1)) \\
    &= \ 
    \tr_{U(1,1)}^{U(2)}((e_1\times 1)\cdot \res^{U(2)}_{U(1,1)}(1-\tr_N^{U(2)}(1))) \ = \ 0 \ .   
  \end{align*}
  One can also show that the class $1-\tr_N^{U(2)}(1)$ is infinitely divisible by the Euler class $e_2=c_2^{(2)}$;
  so it is also in the kernel of the completion map at the ideal $(e_2)$.
 \end{rk}

The Chern class  $c_k^{(m)}$ is defined as a transfer; so identifying its restriction
to a subgroup of $U(m)$ involves a double coset formula.
The following double coset formula will take care of all cases we need in this paper;
it ought to be well-known to experts, but I do not know a reference.
The case $l=1$ is established in \cite[Lemma 4.2]{symonds-splitting},
see also \cite[Example 3.4.13]{schwede:global}.
The double coset space $U(i,j)\bs U(m)/ U(k,l)$ is discussed at various places in the
literature, for example \cite[Example 3]{matsuki:double_coset},
but I have not seen the resulting double coset formula spelled out.

\begin{prop}[Double coset formula]\label{prop:double coset}
  Let $i,j,k,l$ be positive natural numbers such that $i+j=k+l$. Then
  \[ \res^{U(i+j)}_{U(i,j)}\circ\tr_{U(k,l)}^{U(k+l)} \ =
    \sum_{0,k-j\leq d\leq i,k}\, \tr_{U(d,i-d,k-d,j-k+d)}^{U(i,j)}\circ\gamma_d^*\circ \res^{U(k,l)}_{U(d,k-d,i-d,l-i+d)}\ ,\]
  where $\gamma_d\in U(i+j)$ is the permutation matrix of the shuffle permutation $\chi_d\in\Sigma_{i+j}$
  given by
\[ \chi_d(a) \ = \
  \begin{cases}
    a & \text{ for $1\leq a\leq d$,}\\
    a-d+i& \text{ for $d+1\leq a\leq k$,}\\
    a+d-k& \text{ for $k+1\leq a\leq k+i-d$, and}\\
    a & \text{ for $a > k+i-d$.}
  \end{cases}
\]
\end{prop}
\begin{proof}
  We refer to \cite[IV Theorem 6.3]{lms} or \cite[Theorem 3.4.9]{schwede:global} for the general
  double coset formula for $\res^G_K\circ\tr_H^G$ for two closed subgroups $H$ and $K$
  of a compact Lie group $G$; we need to specialize it to the situation at hand.
  We first consider a matrix $A\in U(m)$ such that the center $Z$ of $U(i,j)$
  is {\em not} contained in the $U(i,j)$-stabilizer 
  \[ S_A\ = \ U(i,j)\cap {^A U(k,l)} \]
  of the coset $A\cdot U(k,l)$.
  Then $S_A\cap Z$ is a proper subgroup of the center $Z$ of $U(i,j)$, which is isomorphic
  to $U(1)\times U(1)$. So $S_A\cap Z$ has strictly smaller dimension than $Z$.
  Since the center of $U(i,j)$ is contained in the normalizer of $S_A$,
  we conclude that the group $S_A$ has an infinite Weyl group inside $U(i,j)$.
  All summands in the double coset formula indexed by such points then involve transfers with infinite Weyl groups,
  and hence they vanish.
  
  So all nontrivial contributions to the double coset formula
  stem from double cosets $U(i,j)\cdot A\cdot U(k,l)$ such that $S_A$ contains the center of $U(i,j)$.
  In particular, the matrix
  $ \left(  \begin{smallmatrix}  - E_i & 0 \\ 0 & E_j  \end{smallmatrix} \right)$ then belongs to $S_A$.
  We write $L=A\cdot (\mC^k\oplus 0^l)$, a complex $k$-plane in $\mC^{k+l}$;
  we consider $x\in\mC^i$ and $y\in\mC^j$ such that $(x,y)\in L$.
  Because $ \left(  \begin{smallmatrix}  - E_i & 0 \\ 0 & E_j  \end{smallmatrix} \right)\cdot L=L$,
  we deduce that $(-x,y)\in L$. Since $(x,y)$ and $(-x,y)$ belong to $L$, so do the vectors $(x,0)$ and $(y,0)$.
  We have thus shown that the $k$-plane $L=A\cdot(\mC^k\oplus 0^l)$ is spanned by the intersections
  \[   L\cap (\mC^i\oplus 0^j)   \text{\qquad and\qquad}  L\cap (0^i\oplus\mC^j)\ .  \]  
  We organize the cosets with this property by the dimension of the first intersection:
  we define $M_d$ as the closed subspace of $U(m)/U(k,l)$
  consisting of those cosets $A\cdot U(k,l)$ such that
  \[  \dim_\mC( L\cap (\mC^i\oplus 0^j))\  = \ d
    \text{\qquad and\qquad}
    \dim_\mC( L\cap (0^i\oplus\mC^j))\  = \ k-d\ .  \]
  If $M_d$ is nonempty, we must have $0, k-j\leq d\leq i,k$.
  The group $U(i,j)$ acts transitively on $M_d$, and the coset $\gamma_d\cdot U(k,l)$  belongs to $M_d$;
  so $M_d$ is the $U(i,j)$-orbit type manifold of $U(m)/U(k,l)$ for the conjugacy class of
  \[   S_{\gamma_d}\ =  \ U(i,j)\cap {^{\gamma_d} U(k,l)} \ = \ U(d,i-d,k-d,j-k+d)\ . \]
  The corresponding orbit space $U(i,j)\backslash M_d=U(i,j)\cdot\gamma_d\cdot U(k,l)$
  is a single point inside the double coset space,  so its internal Euler characteristic is 1.
  This orbit type thus contributes the summand
  \[ \tr_{U(d,i-d,k-d,j-k+d)}^{U(i,j)}\circ\gamma_d^*\circ \res^{U(k,l)}_{U(d,k-d,i-d,l-i+d)} \]
  to the double coset formula.
\end{proof}

In \cite[Corollary 8.3]{conner-floyd:relation_cobordism},
Conner and Floyd define Chern classes for complex vector bundles
in the nonequivariant $\bMU$-cohomology rings.
In the universal cases, these yield classes $c_k\in \bMU^{2 k}(B U(m))$
that are nowadays referred to as Conner--Floyd--Chern classes. 
The next theorem spells out the key properties of our Chern classes $c_k^{(m)}$;
parts (i), (ii) and (iii) roughly say that all the familiar structural properties
of the Conner--Floyd--Chern classes in $\bMU^*(B U(m))$
already hold for our Chern classes in $U(m)$-equivariant $\bMU$-theory.
Part (iv) of the theorem refers to the bundling maps $\bMU_G^*\to \bMU^*(B G)$
defined by tom Dieck in \cite[Proposition 1.2]{tomDieck:bordism_integrality}.

\begin{thm}\label{thm:CFC main} The Chern classes in homotopical equivariant bordism enjoy the following properties.
  \begin{enumerate}[\em (i)]
  \item For all $0\leq k\leq m=i+j$, the relation
    \[ \res^{U(m)}_{U(i,j)}(c_k^{(m)})\ = \ \sum_{d=0,\dots,k} c_d^{(i)}\times c_{k-d}^{(j)}\]
    holds in the group $\bMU_{U(i,j)}^{2 k}$.
  \item The relation
    \[ \res^{U(m)}_{U(m-1)}(c_k^{(m)})\ = \
      \begin{cases}
        c_k^{(m-1)}  & \text{ for $0\leq k\leq m-1$, and}\\
      \  0  & \text{ for $k=m$}
      \end{cases}\]
    holds in the group $\bMU_{U(m-1)}^{2 k}$.
\item Let $T^m$ denote the diagonal maximal torus of $U(m)$. Then the restriction homomorphism
    \[ \res^{U(m)}_{T^m} \ : \ \bMU_{U(m)}^{2 k} \ \to \ \bMU^{2 k}_{T^m} \]
    takes the class $c_k^{(m)}$ to the $k$-th elementary symmetric polynomial
    in the classes $p_1^*(e_1)$, \dots, $p_m^*(e_1)$,
    where $p_i:T^m\to T=U(1)$ is the projection to the $i$-th factor.
  \item The bundling map
    \[ \bMU_{U(m)}^* \ \to \ \bMU^*(BU(m)) \]
    takes the class $c_k^{(m)}$ to the $k$-th Conner--Floyd--Chern class.
  \end{enumerate}
\end{thm}
\begin{proof}
  (i) This property exploits the double coset formula for 
  $\res^{U(m)}_{U(i,j)}\circ\tr_{U(k,m-k)}^{U(m)}$ recorded in Proposition \ref{prop:double coset},
  which is the second equation in the following list:
  \begin{align*}
    \res^{U(m)}_{U(i,j)}(c_k^{(m)})\
    &= \ \res^{U(m)}_{U(i,j)}(\tr_{U(k,m-k)}^{U(m)}(e_k\times 1_{m-k}))  \\
    &= \
    \sum_{d=0,\dots,k} \tr_{U(d,i-d,k-d,j-k+d)}^{U(i,j)}(\gamma_d^*(\res^{U(k,m-k)}_{U(d,k-d,i-d,j-k+d)}(e_k\times 1_{m-k})))\\
    &= \
      \sum_{d=0,\dots,k} \tr_{U(d,i-d,k-d,j-k+d)}^{U(i,j)}(\gamma_d^*(e_d\times e_{k-d}\times 1_{i-d}\times 1_{j-k+d}))\\
    &= \
          \sum_{d=0,\dots,k} \tr_{U(d,i-d,k-d,j-k+d)}^{U(i,j)}(e_d\times 1_{i-d}\times e_{k-d}\times 1_{j-k+d})\\
    &= \
      \sum_{d=0,\dots,k} \tr_{U(d,i-d)}^{U(i)}(e_d\times 1_{i-d})\times
\tr_{U(k-d,j-k+d)}^{U(j)}(e_{k-d}\times 1_{j-k+d})\\
    &= \  \sum_{d=0,\dots,k} c_d^{(i)}\times c_{k-d}^{(j)}
  \end{align*}

Part (ii) for $k<m$ follows from part (i) by restriction from $U(m-1,1)$ to $U(m-1)$:
\begin{align*}
  \res^{U(m)}_{U(m-1)}(c_k^{(m)})\
  &= \ \res^{U(m-1,1)}_{U(m-1)}(\res^{U(m)}_{U(m-1,1)}(c_k^{(m)}))\\
  &= \  \res^{U(m-1,1)}_{U(m-1)}(c_{k-1}^{(m-1)}\times c_1^{(1)}\ + \ c_k^{(m-1)}\times c_0^{(1)})\\
    &= \ c_{k-1}^{(m-1)}\times \res^{U(1)}_1(c_1^{(1)})\ +\  c_k^{(m-1)}\times \res^{U(1)}_1(c_0^{(1)})\ = \ c_k^{(m-1)}\ .   
\end{align*}
We have used that the class $c_1^{(1)}=e_1$ is in the kernel of the augmentation
$\res^{U(1)}_1:\bMU_{U(1)}^*\to \bMU^*$. The Euler class $c_m^{(m)}=e(\nu_m)$
restricts to 0 in $\bMU^*_{U(m-1)}$ because the restriction 
of the tautological $U(m)$-representation to $U(m-1)$ splits off a trivial one-dimensional summand.

  (iii) An inductive argument based on property (i) shows the desired relation:
  \begin{align*}
     \res^{U(m)}_{T^m}(c_k^{(m)}) \ &= \
    \res^{U(m)}_{U(1,\dots,1)}(c_k^{(m)}) \\
    &= \ \sum_{A\subset\{1,\dots,m\}, |A|=k}\quad \prod_{a\in A} p_a^*(c_1^{(1)})\cdot\prod_{b\not\in A}p_b^*(c_0^{(1)}) \\
    &= \ \sum_{A\subset\{1,\dots,m\}, |A|=k} \quad \prod_{a\in A} p_a^*(e_1)\ .
  \end{align*}

(iv) 
  As before we let $T^m$ denote the diagonal maximal torus in $U(m)$.
  The splitting principle holds for nonequivariant complex oriented cohomology theories,
  see for example \cite[Proposition 8.10]{dold:Chern_classes}.
  In other words, the right vertical map in the commutative square of graded rings is injective:
  \[ \xymatrix{ \bMU^*_{U(m)}\ar[r]\ar[d]_{\res^{U(m)}_{T^m}} & \bMU^*(B U(m))\ar[d]^-{ (B i)^*} \\
      \bMU^*_{T^m}\ar[r] & \bMU^*(B T^m) \ar@{=}[r] & \bMU^*[[p_1^*(e_1),\dots,p_m^*(e_1)]]
    } \]
  The $k$-th Conner--Floyd--Chern class is characterized as the unique element
  of $\bMU^{2 k}(B U(m))$ that maps to
  the $k$-th elementary symmetric polynomial in the classes $p_1^*(e_1),\dots,p_m^*(e_1)$.
  Together with part (iii), this proves the claim.
\end{proof}

\section{Regularity results}

In this section we use the Chern classes to formulate new structural properties of the equivariant
bordism ring $\bMU_{U(m)}^*$. In particular, we can say what $\bMU_{U(m)}^*$
looks like after dividing out some of the Chern classes, and after completing at the Chern classes.
The following theorem states these facts more generally for $U(m)\times G$
instead of $U(m)$; by induction on the number of factors, we can then deduce corresponding results
for products of unitary groups, see Theorem \ref{thm:completions}.
The results in this section make crucial use of the splitting theorem for
global functors established in \cite{schwede:split BU}.

\begin{thm}\label{thm:structure}
  For every compact Lie group $G$ and all $0\leq k\leq m$,
  the sequence of Chern classes\vspace*{-.1cm}
  \[ (c_m^{(m)}\times 1_G,\ c_{m-1}^{(m)}\times 1_G,\dots,\ c_{k+1}^{(m)}\times 1_G)  \]
  is a regular sequence in the graded-commutative ring  $\bMU^*_{U(m)\times G}$
  that generates the kernel of the surjective restriction homomorphism\vspace*{-.1cm}
  \[ \res^{ U(m)\times G}_{U(k)\times G}\ :\ \bMU_{U(m)\times G}^*\ \to\ \bMU_{U(k)\times G}^*\ . \]
  In particular, the sequence of Chern classes $(c_m^{(m)},c_{m-1}^{(m)},\dots,c_1^{(m)})$
  is a regular sequence   that generates the augmentation ideal of the graded-commutative ring  $\bMU^*_{U(m)}$.
\end{thm}
\begin{proof}
  We argue by downward induction on $k$. The induction starts with $k=m$, where there is nothing to show.
  Now we assume the claim for some $k\leq m$, and we deduce it for $k-1$.
  The inductive hypothesis shows that
  $c_m^{(m)}\times 1_G,\dots,c_{k+1}^{(m)}\times 1_G$
  is a regular sequence in the graded-commutative ring  $\bMU^*_{U(m)\times G}$,
  and that the restriction homomorphism $\res^{U(m)\times G}_{U(k)\times G}$
  factors through an isomorphism
  \[ \bMU_{U(m)\times G}^*/(c_m^{(m)}\times 1_G,\dots,c_{k+1}^{(m)}\times 1_G)
    \ \iso\ \bMU_{U(k)\times G}^*\ . \]
  We exploit that the various equivariant bordism spectra $\bMU_G$ underlie a global spectrum,
  see \cite[Example 6.1.53]{schwede:global};
  thus the restriction homomorphism $\res^{U(k)\times G}_{U(k-1)\times G}$ is surjective by
  Theorem 1.4 and Proposition 2.2 of \cite{schwede:split BU}.
  Hence the  standard long exact sequence unsplices into
  a short exact sequence of graded $\bMU^*$-modules:\vspace*{-.2cm}
  \[ 0\  \to\ \bMU_{U(k)\times G}^{*-2 k}\  \xra{(e_k\times 1_G)\cdot -\ }\
    \bMU_{U(k)\times G}^* \xra{\res^{U(k)\times G}_{U(k-1)\times G}}\ \bMU_{U(k-1)\times G}^*\ \to\ 0 
  \]
  Because
  \[ \res^{U(m)\times G}_{U(k)\times G}(c_k^{(m)}\times 1_G)\ = \   c_k^{(k)}\times 1_G\ = \ e_k\times 1_G\ , \]
  we conclude that $c_k^{(m)}\times 1_G$ is a non zero-divisor in 
  $\bMU_{U(m)\times G}^*/(c_m^{(m)}\times 1_G,c_{m-1}^{(m)}\times 1_G,\dots,c_{k+1}^{(m)}\times 1_G)$,
  and that additionally dividing out $c_k^{(m)}\times 1_G$ yields $\bMU_{U(k-1)\times G}^*$.
  This completes the inductive step.
\end{proof}

We can now identify the completion of $\bMU^*_{U(m)}$ at the augmentation ideal
as an $\bMU^*$-power series algebra on the Chern classes.
We state this somewhat more generally for products of unitary groups, which we write as
\[ U(m_1,\dots,m_l)\ = \ U(m_1)\times\dots\times U(m_l)\  , \]
for natural numbers $m_1,\dots,m_l\geq 1$.
For $1\leq i\leq l$, we write
\[ p_i\ :\ U(m_1,\dots,m_l)\ \to\ U(m_i)  \]
for the projection to the $i$-th factor, and we set\vspace*{-.2cm}
\[ c^{[i]}_k \ = \ p_i^*(c_k^{(m_i)})\ = \ 1_{U(m_1,\dots,m_{i-1})}\times c_k^{(m_i)}\times 1_{U(m_{i+1},\dots,m_l)}
  \ \in \ \bMU_{U(m_1,\dots,m_l)}^{2 k}\ .\]
The following theorem was previously known for tori, that is, for $m_1=\dots=m_l=1$.

\begin{thm}\label{thm:completions}
  Let $m_1,\dots,m_l\geq 1$ be positive integers.
  \begin{enumerate}[\em (i)]
  \item
    The sequence of Chern classes\vspace*{-.2cm}
    \begin{equation}\label{eq:Chern_for_products}
       c_{m_1}^{[1]},\dots,c_1^{[1]},c_{m_2}^{[2]},\dots,c_1^{[2]},\dots, c_{m_l}^{[l]},\dots,c_1^{[l]}
    \end{equation}
  is a regular sequence that generates the augmentation ideal of the graded-commutative ring  $\bMU^*_{U(m_1,\dots,m_l)}$.
 \item
   The completion of $\bMU_{U(m_1,\dots,m_l)}^*$ at the augmentation ideal
   is a graded $\bMU^*$-power series algebra in the Chern classes \eqref{eq:Chern_for_products}.
 \item
    The bundling map $\bMU_{U(m_1,\dots,m_l)}^*\to \bMU^*(B U(m_1,\dots,m_l))$ extends to an isomorphism
    \[ ( \bMU_{U(m_1,\dots,m_l)}^*)^\wedge_I \ \to \ \bMU^*(BU(m_1,\dots,m_l)) \]
    from the completion at the augmentation ideal.
  \end{enumerate}
\end{thm}
\begin{proof}
  Part (i) follows from Theorem \ref{thm:structure} by induction on the number $l$ of factors.

  We prove parts (ii) and (iii) together.
  We must show that for every $n\geq 1$,  $\bMU_{U(m_1,\dots,m_l)}^*/I^n$ is free as an $\bMU^*$-module on the monomials
  of degree less than $n$ in the Chern classes \eqref{eq:Chern_for_products}.
  There is nothing to show for $n=1$.
  The short exact sequence
  \[ 0\ \to\ I^n/I^{n+1}\ \to\ \bMU_{U(m_1,\dots,m_l)}^*/I^{n+1}\ \to\ \bMU_{U(m_1,\dots,m_l)}^*/I^n\ \to\ 0\]
  and the inductive hypothesis reduce the claim to showing that
  $I^n/I^{n+1}$ is free as an $\bMU^*$-module on the monomials of degree exactly $n$ in the Chern classes
  \eqref{eq:Chern_for_products}.
  Since the augmentation ideal $I$ is generated by these Chern classes,
  the $n$-th power $I^n$ is generated, as a module over $\bMU_{U(m_1,\dots,m_l)}^*$, by the monomials
  of degree $n$.
  So $I^n/I^{n+1}$ is generated by these monomials as a module over $\bMU^*$. 
  
  The bundling map $\bMU_{U(m_1,\dots,m_l)}^*\to \bMU^*(B U(m_1,\dots,m_l))$
  is a homomorphism of augmented $\bMU^*$-algebras,
  and it takes the Chern class $c_k^{[i]}$ to the inflation of the $k$-th Conner--Floyd--Chern class
  along the projection to the $i$-th factor.
  By the theory of complex orientations, the collection of these
  Conner--Floyd--Chern classes are $\bMU^*$-power series generators of $\bMU^*(B U(m_1,\dots,m_l))$;
  in particular,
  the images of the Chern class monomials are $\bMU^*$-linearly independent in $\bMU^*(B U(m_1,\dots,m_l))$.
  Hence these classes are themselves linearly independent in $I^n/I^{n+1}$.
\end{proof}

\begin{rk}\label{rk:degenerate}
Greenlees and May \cite[Corollary 1.6]{greenlees-may:completion} construct
a local homology spectral sequence
\[ E_2^{p,q}\ = \ H^I_{-p,-p}(\bMU_G^*)\  \Longrightarrow \ \bMU^{p+q}( B G )\ .\]
The regularity results about Chern classes from Theorem \ref{thm:completions} imply that
whenever $G=U(m_1,\dots,m_l)$ is a product of unitary groups, the $E_2^{p,q}$-term vanishes for all $p\ne 0$,
and the spectral sequence degenerates into the isomorphism
\[ E_2^{0,*}\ = \ (\bMU_{U(m_1,\dots,m_l)}^*)^\wedge_I \  \iso \ \bMU^*( B U(m_1,\dots,m_l)) \]
of Theorem \ref{thm:completions} (iii).
\end{rk}

\begin{rk}
  The previous regularity theorems are special cases of the following more general results
  that hold for every global $\bMU$-module $E$:
    \begin{itemize}
  \item  For every compact Lie group $G$, the sequence of Chern classes 
    $c_m^{(m)}\times 1_G,\dots,c_1^{(m)}\times 1_G$
    acts regularly  on the graded $\bMU^*_{U(m)\times G}$-module $E^*_{U(m)\times G}$.
\item The restriction homomorphism
  \[ \res^{ U(m)\times G}_ G\ :\ E_{U(m)\times G}^*\ \to \ E_G^*\]
  factors through an isomorphism
  \[  E_{U(m)\times G}^*/(c_m^{(m)}\times 1_G,\dots, c_1^{(m)}\times 1_G)\ \iso \ E_G^* \ .\]
  
\item
  For all $m_1,\dots,m_l\geq 1$, the sequence of Chern classes \eqref{eq:Chern_for_products}
  acts regularly on the graded $\bMU^*_{U(m_1,\dots,m_l)}$-module $E^*_{U(m_1,\dots,m_l)}$.
  \end{itemize}
  As in Remark \ref{rk:degenerate}, the regularity properties also imply the degeneracy
  of the Greenlees--May local homology spectral sequence converging to $E^*(B U(m_1,\dots,m_l))$.
\end{rk}

\section{The \texorpdfstring{$\bMU$}{MU}-completion theorem via Chern classes}

In this section we use the Chern classes to reformulate the $\bMU_G$-comple\-tion theorem
of Greenlees--May \cite{greenlees-may:completion} and La Vecchia \cite{lavecchia},
for any compact Lie group $G$, and we give a short and self-contained proof.
We emphasize that the essential arguments of this section are all contained in
\cite{greenlees-may:completion} and \cite{lavecchia};
the Chern classes let us arrange them in a more conceptual and concise way.
The references \cite{greenlees-may:completion, lavecchia}
ask for a finitely generated ideal of $\bMU_G^*$ that is `sufficiently large' in the
sense of \cite[Definition 2.4]{greenlees-may:completion}.
While we have no need to explicitly mention sufficiently large ideals,
the new insight is that the ideal generated by the Chern classes
of any faithful $G$-representation is `sufficiently large'.

\begin{con}[Chern classes of representations]
  We let $V$ be a complex representation of a compact Lie group $G$.
  We let $\rho:G\to U(m)$ be a continuous homomorphism that classifies $V$,
  that is, such that $\rho^*(\nu_m)$ is isomorphic to $V$; here, $m=\dim_\mC(V)$.
  The {\em $k$-th Chern class} of $V$ is
  \[ c_k(V)\ = \ \rho^*(c_k^{(m)})\ \in \ \bMU_G^{2 k}\ .\]
  In particular, $c_0(V)=1$, $c_m(V)=e(V)$ is the Euler class, and $c_k(V)=0$ for $k>m$.
\end{con}

\begin{eg} As an example, we consider the tautological representation $\nu_2$ of $S U(2)$ on $\mC^2$.
  By the general properties of Chern classes we have
  $c_0(\nu_2)=1$, $c_2(\nu_2)=e(\nu_2)$ is the Euler class,
  and $c_k(\nu_2)=0$ for $k\geq 3$. The first Chern class of $\nu_2$ can be rewritten
  by using a double coset formula as follows:
  \begin{align*}
    c_1(\nu_2)\
    &= \ \res^{U(2)}_{S U(2)}(c_1^{(2)}) \ = \ \res^{U(2)}_{S U(2)}(\tr_{U(1,1)}^{U(2)}(e_1\times 1)) \\
    &= \ \tr_T^{S U(2)}(\res^{U(1,1)}_T(e_1\times 1)) \ = \ \tr_T^{S U(2)}(e(\chi)) \ .
  \end{align*}
Here $T=\{
(\begin{smallmatrix} \lambda & 0 \\ 0 & \lambda^{-1} \end{smallmatrix}) \ : \ \lambda\in U(1)
\}$
is the diagonal maximal torus of $S U(2)$, $\chi:T\iso U(1)$ is the character that projects
onto the upper left diagonal entry, and $e(\chi)\in\bMU^2_T$ is its Euler class.
\end{eg}

\begin{con}
  We recall a specific $G$-equivariant $\bMU_G$-module associated
  to a complex representation $V$ of a compact Lie group $G$.
  The construction is known as the {\em stable Koszul complex}
  for the Chern classes $c_1(V),\dots,c_m(V)$, where $m=\dim_\mC(V)$;
  in the notation of Greenlees--May \cite{greenlees-may:completion}
  and La Vecchia \cite{lavecchia}, our $K(G,V)$ would appear as $\Gamma_I(\bMU_G)$,
  where $I=(c_1(V),\dots,c_m(V))$ is the ideal generated by the Chern classes.

  For any equivariant homotopy class $x\in \bMU_G^l$,
  we write $\bMU_G[1/x]$ for the $\bMU_G$-module localization of $\bMU_G$ with $x$ inverted;
  in other words, $\bMU_G[1/x]$ is a homotopy colimit (mapping telescope)
  in the triangulated category of the sequence
\[ \bMU_G\ \xra{-\cdot x} \ \Sigma^l \bMU_G\ \xra{-\cdot x} \Sigma^{2 l}\bMU_G\ \xra{-\cdot x} \ \Sigma^{3 l}\bMU_G\ \xra{-\cdot x} \ \dots \ \ .\]
  We write $K(x)$ for the fiber of the morphism $\bMU_G\to\bMU_G[1/x]$.
  Then we define
  \[ K(G,V)\ = \ K(c_1(V))\sm_{\bMU_G}\dots \sm_{\bMU_G}K(c_m(V))\ . \]
  The smash product of the morphisms $K(c_i(V))\to\bMU_G$ provides a morphism of $G$-equivariant $\bMU_G$-module spectra
  \[ \epsilon_V\ : \ K(G,V)\ \to \ \bMU_G .\]
  By general principles, the module $K(G,V)$
  only depends on the radical of the ideal generated by the classes $c_1(V),\dots,c_m(V)$.
  But more is true: As a consequence of Theorem \ref{thm:completion} below,
  $K(G,V)$ is entirely independent, as a $G$-equivariant $\bMU_G$-module, of the faithful representation $V$.
\end{con}

\begin{prop}\label{prop:characterize K(G,V)}
  Let $V$ be a faithful complex representation of a compact Lie group $G$.
  \begin{enumerate}[\em (i)]
  \item  The morphism
    $\epsilon_V:K(G,V)\to \bMU_G$ is an equivalence of underlying nonequivariant spectra.
  \item For every nontrivial closed subgroup $H$ of $G$, the $H$-geometric fixed point spectrum
    $\Phi^H(K(G,V))$ is trivial.
\end{enumerate}
\end{prop}
\begin{proof}
  (i) We set $m=\dim_\mC(V)$.
  The Chern classes $c_1(V),\dots,c_m(V)$ belong to the augmentation ideal
  of $\bMU_G^*$, so they restrict to 0 in $\bMU_{\{1\}}^*$, and
  hence the underlying nonequivariant spectrum of $\bMU_G[1/c_i(V)]$ is trivial
  for each $i=1,\dots,m$.
  Hence the morphisms $K(c_i(V))\to \bMU_G$ are underlying nonequivariant
  equivalences for $i=1,\dots,m$.
  So also the morphism $\epsilon_V$ is an  underlying nonequivariant equivalence.

  (ii) We let $H$ be a nontrivial closed subgroup of $G$.
  We set $W=V-V^H$, the orthogonal complement of the $H$-fixed points.
  This is a complex $H$-representation with $W^H=0$; moreover, $W$ is nonzero because
  $H$ acts faithfully on $V$ and $H\ne\{1\}$.  
  For $k=\dim_\mC(W)$ we then have
  \[ e(W) \ = \ c_k(W)\ = \ c_k(W\oplus V^H)\ = \ c_k(\res^G_H(V))\ = \  \res^G_H( c_k(V)) \ ;\]
  the second equation uses the fact that adding a trivial representation
  leaves Chern classes unchanged, by part (ii) of Theorem \ref{thm:CFC main}.

  Since $W^H=0$, the geometric fixed point homomorphism $\Phi^H:\bMU_H^*\to \Phi_H^*(\bMU)$
  sends the Euler class $e(W)  = \res^G_H( c_k(V))$ to an invertible element.
  The functor $\Phi^H\circ \res^G_H$ commutes with inverting elements.
  Since the class $\Phi^H(\res^G_H(c_k(V)))$ is already invertible,
  the localization morphism $\bMU_G\to \bMU_G[1/c_k(V)]$ induces an equivalence on $H$-geometric fixed points.
  Since the functor $\Phi^H\circ\res^G_H$ is exact, it annihilates the fiber $K(c_k(V))$
  of the localization $\bMU_G\to \bMU_G[1/c_k(V)]$.
  The functor $\Phi^H\circ\res^G_H$ is also strong monoidal, in the sense of a natural
  equivalence of nonequivariant spectra
  \[ \Phi^H(X\sm_{\bMU_G}Y) \ \simeq \ \Phi^H(X)\sm_{\Phi^H(\bMU_G)}\Phi^H(Y) \ ,  \]
  for all $G$-equivariant $\bMU_G$-modules $X$ and $Y$.
  Since $K(G,V)$ contains $K(c_k(V))$ as a factor (with respect to $\sm_{\bMU_G}$),
  we conclude that the spectrum $\Phi^H(K(G,V))$ is trivial.
\end{proof}

The following 'completion theorem' is a reformulation of the combined
work of Greenlees--May \cite[Theorem 1.3]{greenlees-may:completion}
and La Vecchia \cite{lavecchia}. It is somewhat more precise in that an unspecified
`sufficiently large' finitely generated ideal of $\bMU_G^*$ is replaced by the
ideal generated by the Chern classes of a faithful $G$-representation.
The proof is immediate from the properties of $K(G,V)$ listed in
Proposition \ref{prop:characterize K(G,V)}.
We emphasize, however, that our proof is just a different way of arranging some arguments from
\cite{greenlees-may:completion} and \cite{lavecchia} while taking advantage of the Chern class formalism.

Since the morphism $\epsilon_V:K(G,V)\to \bMU_G$ is a nonequivariant equivalence
of underlying spectra, the morphism $E G_+\sm \bMU_G\to \bMU_G$ that collapses
the universal space $E G$ to a point admits a unique lift to a morphism
of $G$-equivariant $\bMU_G$-modules $\psi: E G_+\sm \bMU_G\to K(G,V)$ across $\epsilon_V$.

\begin{thm}\label{thm:completion}
  Let $V$ be a faithful complex representation of a compact Lie group $G$.
  Then the morphism
  \[ \psi\ : \ E G_+\sm \bMU_G\ \to\ K(G,V) \]
  is an equivalence of $G$-equivariant $\bMU_G$-module spectra.
\end{thm}
\begin{proof}
  Because the underlying space of $E G$ is contractible, the composite 
  \[ E G_+\sm \bMU_G\ \xra{\ \psi\ } \ K(G,V)\ \xra{\ \epsilon_V\ }\ \bMU_G \]
  is an equivariant equivalence of underlying nonequivariant spectra.
  Since $\epsilon_V$ is an equivariant equivalence of underlying nonequivariant spectra
  by Proposition \ref{prop:characterize K(G,V)}, so is $\psi$.
  For all nontrivial closed subgroups $H$ of $G$, source and target of $\psi$
  have trivial $H$-geometric fixed points spectra,
  again by Proposition \ref{prop:characterize K(G,V)}. So the morphism $\psi$ induces
  an equivalence on geometric fixed point spectra for all closed subgroup of $G$,
  and it is thus an equivariant equivalence.
\end{proof}

\vspace*{-.1cm}
{\bf Acknowledgments.}
The author is a member of the Hausdorff Center for Mathematics
at the University of Bonn (DFG GZ 2047/1, project ID 390685813).

\vspace*{-.1cm}

\end{document}